\documentclass[12pt]{amsart}
\usepackage{amssymb,amscd}
\usepackage{verbatim}

\usepackage{amsmath,amssymb,graphicx,mathrsfs}   
\usepackage{enumerate}
\usepackage[colorlinks=true,allcolors = blue]{hyperref} 

\usepackage{tikz}
\usetikzlibrary{matrix}

\usepackage[all]{xy}

\let\frak\mathfrak

\def\>{\relax\ifmmode\mskip.666667\thinmuskip\relax\else\kern.111111em\fi}
\def\<{\relax\ifmmode\mskip-.333333\thinmuskip\relax\else\kern-.0555556em\fi}
\def\vsk#1>{\vskip#1\baselineskip}
\def\vv#1>{\vadjust{\vsk#1>}\ignorespaces}
\def\vvn#1>{\vadjust{\nobreak\vsk#1>\nobreak}\ignorespaces}

  \let\ssize\scriptstyle
\let\sssize\scriptscriptstyle

\let\Medskip\medskip
\def\medskip{\par\Medskip}
\let\Bigskip\bigskip
\def\bigskip{\par\Bigskip}

\let\Maketitle\maketitle
\def\maketitle{\Maketitle\thispagestyle{empty}\let\maketitle\empty}

\newtheorem{thm}{Theorem}[section]
\newtheorem{cor}[thm]{Corollary}
\newtheorem{lem}[thm]{Lemma}

\theoremstyle{definition}                                  
\numberwithin{equation}{section}

\theoremstyle{definition}

\let\mc\mathcal
\let\nc\newcommand

\let\la\lambda
\let\La\Lambda

\let\phi\varphi

\let\der\partial

\let\geq\geqslant

\let\leq\leqslant

\let\on\operatorname
\let\bi\bibitem
\let\bs\boldsymbol

\def\C{{\mathbb C}}
\def\Z{{\mathbb Z}}

\def\F{{\mathbb F}}   

\def\+#1{^{\{#1\}}}

\def\beq{\begin{equation}}
\def\eeq{\end{equation}}
\def\be{\begin{equation*}}
\def\ee{\end{equation*}}

\nc{\bea}{\begin{eqnarray*}}
\nc{\eea}{\end{eqnarray*}}
\nc{\bean}{\begin{eqnarray}}
\nc{\eean}{\end{eqnarray}}

\nc{\Il}{{\mc I_{\bs\la}}}
\nc{\bla}{{\bs\la}}
\nc{\Fla}{\F_\bla}
\nc{\tfl}{{T^*\Fla}}
\nc{\GL}{{GL_n(\C)}}
\nc{\GLC}{{GL_n(\C)\times\C^*}}

\let\sd s 

\def\ddk_#1{\kk_{#1}\<\>\frac\der{\der\<\>\kk_{#1}}}

\def\bul{\mathbin{\raise.2ex\hbox{$\sssize\bullet$}}}
\def\intt{\mathchoice
{\mathop{\raise.2ex\rlap{$\,\,\ssize\backslash$}{\intop}}\nolimits}
{\mathop{\raise.3ex\rlap{$\,\sssize\backslash$}{\intop}}\nolimits}
{\mathop{\raise.1ex\rlap{$\sssize\>\backslash$}{\intop}}\nolimits}
{\mathop{\rlap{$\sssize\<\>\backslash$}{\intop}}\nolimits}}

\let\kk q 
\let\cc c

\let\Ko K

\def\GZ/{Gelfand-Zetlin}
\def\KZ/{{\slshape KZ\/}}
\def\qKZ/{{\slshape qKZ\/}}
\def\XXX/{{\slshape XXX\/}}

\def\zz{{\bs z}}

\nc{\A}{{\mc A}}

\def\Q{{\mathbb Q}}

\nc{\hsl}{\widehat{{\frak{sl}_2}}}

\nc{\BC}{{ \mathbb C}}
\nc{\lra}{\longrightarrow}
\nc{\CO}{{\mathcal{O}}}
\nc{\BZ}{{ \mathbb Z}}
\nc{\hfn}{\hat{\frak{n}}}
\nc\Zs{{\Z/p^s\Z}}
\nc\Zo{{\Zs[z]^0}}
\nc\gr{{\on{gr}}}

\nc\fD{{\frak D}}

\def\aa{{\bs a}}
\def\grad{{\on{grad}}}

\begin{document}

\hrule width0pt
\vsk->

\title[Dynamical and \qKZ/ equations modulo $p^s$, an example]
{Dynamical and \qKZ/ equations modulo $p^s$,
\\
 an example}

\author[Alexander Varchenko]
{Alexander Varchenko}

\maketitle

\begin{center}
{\it $^{\star}$ Department of Mathematics, University
of North Carolina at Chapel Hill\\ Chapel Hill, NC 27599-3250, USA\/}



\end{center}

\vsk>
{\leftskip3pc \rightskip\leftskip \parindent0pt \Small
{\it Key words\/}: Dynamical and \qKZ/ equations; $p^s$-hypergeometric solutions; master polynomials; Dwork congruences.

\vsk.6>
{\it 2020 Mathematics Subject Classification\/}: 11D79 (12H25, 32G34, 33C05, 33E30)
\par}

{\let\thefootnote\relax
\footnotetext{\vsk-.8>\noindent
$^\star\<${\sl E\>-mail}:\enspace anv@email.unc.edu,
supported in part by NSF grant DMS-1954266
}}

\begin{abstract}

We consider an example of the joint system of dynamical differential equations and \qKZ/ 
difference equations
with parameters corresponding to equations for elliptic integrals. We solve this system of equations modulo
 any power $p^n$ of a 
 prime integer $p$. We show that the $p$-adic limit of these solutions as $n\to\infty$
determines
a sequence of line bundles, each of which is invariant with respect to the corresponding dynamical connection,
 and that sequence of line bundles is invariant with respect to
  the corresponding \qKZ/ difference connection.

\end{abstract}

{\small\tableofcontents\par}

\setcounter{footnote}{0}
\renewcommand{\thefootnote}{\arabic{footnote}}

\section{Introduction}

\noindent
Let $\zz=(z_1,z_2)$, 
\bea
\Phi (t;\zz;\la, \mu) = t^{-\la}(t-z_1)^{-\mu}
(t-z_2)^{-\mu},
\eea
where $\la, \mu$ are rational numbers.
Consider the column vector 
\bean
\label{IC}
I^{(C)}(\zz;\la,\mu)=\int_C  \Big(\frac {\Phi}{t-z_1}, \frac {\Phi}{t-z_2}\Big)dt ,
\eean
where $C\subset \C-\{0,z_1,z_2\}$  is a contour 
on which the integrand  takes its initial value when $t$ encircles $C$.
As a function of $\zz$, the vector $I^{(C)}(\zz;\la,\mu)$ extends to a multi-valued analytic function on 
$\{\aa \in \C^2\ |\ a_1a_2(a_1-a_2)\ne 0\}$.

\vsk.2>
The function $I^{(C)}(\zz;\la,\mu)$ satisfies the differential and difference equations,
\bean
\label{DY}
z_1\frac{\der I}{\der z_1}(\zz;\la,\mu) &=&
\left(
\begin{bmatrix}
-\la - \mu & -\mu
\\
0&0
\end{bmatrix} 
+\frac{\mu z_1}{z_1-z_2}
 \begin{bmatrix}
-1& 1
\\
1
& -1
\end{bmatrix}  
\right) I(\zz;\la,\mu)  ,
\\ 
\notag
 z_2\frac{\der I}{\der z_2}(\zz;\la,\mu) &=&
\left(
\begin{bmatrix}
0 & 0
\\
-\mu &-\la-\mu
\end{bmatrix} 
+\frac{\mu z_2}{z_2-z_1}
 \begin{bmatrix}
-1& 1
\\
1
& -1
\end{bmatrix}  
\right) I(\zz;\la,\mu) ,
\\
\label{KZ}
I (\zz;\la+1,\mu)
&=&
\begin{bmatrix} \frac{\la+\mu}{z_1\la} &\frac \mu{z_1\la}
\\
\frac \mu{z_2\la}&\frac{\la+\mu}{z_2\la}
\end{bmatrix} I(\zz;\la,\mu)\, .
\eean
If $\la, \mu, \la+2\mu\notin\Z_{>0}$, then all solutions of these equations are given by integrals 
$I^{(C)}(\zz;\la,\mu)$ (with different choices of $C$)
up to multiplication by a scalar 1-periodic function of $\la$, this fact follows from \cite[Theorem 1.1]{V1}.

\vsk.2>
Up to a gauge transformation, equations \eqref{DY}, \eqref{KZ} are the simplest 
example of the trigonometric  \KZ/ differential 
equations and dynamical difference equations, see \cite{TV1,MV}. 
They are also  the simplest example 
of the dynamical differential equations and \qKZ/ difference equations,
see \cite{TV2,TV3}. Up to a gauge transformation, they are the equivariant quantum differential equations and
\qKZ/ difference equation associated with the cotangent bundle of projective line, see \cite{TV3}.
The family of functions $I^{(C)}(\zz;\la,\mu)$, labeled by contours $C$, are the hypergeometric solutions of these
equations constructed in \cite{MV, TV2}. In particular, see  the integral $I^{(C)}(\zz;\la,\mu)$
(gauge transformed) in \cite[Section 7.4]{TV2}. We call equations \eqref{DY} the dynamical differential equations and equation \eqref{KZ} the \qKZ/ difference equation.

\vsk.2>

In this paper we discuss polynomial solutions of equations
 \eqref{DY} and \eqref{KZ} modulo $p^n$, where $p$ is a prime integer and $n$ is a positive integer. 
 We also discuss the $p$-adic limit of these solutions as $n\to \infty$. 
 
 \vsk.2>
 
 More precisely, we consider  the following problem. 
For $\la_0\in\Q$, let 
$\La(\la_0)=\{\la_0 + l\ | \ l\in\Z\}$ be the arithmetic sequence with
initial term $\la_0$ and step $1$. For a positive integer $\ell$, let
$\La(\la_0,\ell)=\{\la_0 + l\ | \ l\in\Z,\, |\la_0+l|< \ell\}$ be an interval of 
the sequence $\La$. 

\medskip

\noindent
{\bf Problem.} Let $p$ be an prime integer, $\la_0,\mu_0\in\Q$, $\ell\in \Z_{>0}$. For any
$n\in\Z_{>0}$, find a sequence of column vectors
\bean
\label{covec} 
I(\zz;\la;\ell;n)=(I_1(\zz;\la;\ell;n), I_2(\zz;\la,\ell;n)),\qquad \la\in \La(\la_0,\ell),
\eean
such that 
\begin{enumerate}

\item[$(\on{i})$]
the coordinates of these vectors are  polynomials in $\zz$ with integer coefficients, 

\item[$(\on{ii})$]

each of the vectors $I(\zz;\la;\ell;n)$ satisfies  modulo $p^n$  differential equations \eqref{DY}
with parameter $\mu=\mu_0$;

\item[$(\on{iii})$]
 
this sequence of vectors satisfies  modulo $p^n$  the difference
equation \eqref{KZ} 
with parameter $\mu=\mu_0$.

\end{enumerate}
We may also require that the vectors  are functorial in the following sense.
If $\la\in \La(\la_0,\ell_0)$ for some $l_0$, then the vector $I(\zz;\la;\ell;n)$ does not depend on $\ell$ for
$\ell \geq\ell_0$. Having a solution $\{I(\zz;\la;\ell;n)\}$ of this problem we may  study the $p$-adic limit of
the vectors  as $n\to\infty$.

\vsk.2>
In this paper we construct a solution of this problem for 
\bean
\label{chp}
&(\la_0, \mu_0) = \big(\frac{1}2, \frac{1}2\big).
\eean
We also describe the $p$-adic limit of our solution. It turns out that the limit is not a solution 
of  equations \eqref{DY} and \eqref{KZ} over a $p$-adic field, as one may naively think, 
but a line bundle invariant with respect to the dynamical connection, defined by equations \eqref{DY},
and invariant with respect to the discrete \qKZ/ connection, defined by equation \eqref{KZ}, see Theorem
\ref{thm inv}.
Notice that there is no such a line bundle if we consider the same differential and difference equations
over the field of complex numbers, see Section \ref{sec last}.

\vsk.2>
The choice of parameters in \eqref{chp} corresponds to elliptic integrals in \eqref{IC}. 
This choice  is technically, 
arithmetically easier than the choice of an arbitrary pair $(\la_0,\mu_0)$ of rational numbers, 
although a similar construction can be performed for a wide class of parameters  $(\la_0,\mu_0)$.

\vsk.2>

Quantum differential equations and associated \qKZ/ difference equations, as well as their solutions, is a
mathematical structure
 with applications in representation theory, algebraic geometry, theory of special function, to name
  a few. It would be interesting to study how the properties of these equations
 and their solutions are reflected in the solutions of the same equations modulo powers of a prime integer
 and  in their $p$-adic limits.

\smallskip
In Section \ref{sec 2} we reformulate equations \eqref{DY} and \eqref{KZ} for 
$(\la_0, \mu_0) = \big(\frac{1}2, \frac{1}2\big)$, see equations \eqref{dy}, \eqref{kz}.
In Section \ref{sec 3} we solve equations \eqref{dy}, \eqref{kz} modulo a power $p^s$ of an
odd prime integer $p$.  The construction of  these solutions is a variant of the constructions
in \cite{SV,V3}.  The constructed solutions are called the $p^s$-hypergeometric solutions in \cite{V3}.
On the $p^s$-hypergeometric solutions see also \cite{V2,V4, V5, VZ1,VZ2, RV1, RV2}.

We prove Dwork-type congruences for these solutions in Section \ref{sec 4}.
Using these congruences we descibe the $p$-adic limit of our solutions as $s\to\infty$ in
 Section \ref{sec 5}.

\smallskip
The author thanks R.\,Rim\'anyi  and A.\,Smirnov for useful discussions and
  Max Planck Institute for Mathematics in Bonn for hospitality in May-June of 2022.

\newpage

\section{Equations for $(\la_0, \mu_0) = \big(\frac{1}2, \frac{1}2\big)$}
\label{sec 2}

Denote 
\bean
\label{H}
&&
H_1(\zz;\la) =
\begin{bmatrix}
-\la - 1 & -1
\\
0&0
\end{bmatrix} 
+\frac{ z_1}{z_1-z_2}
 \begin{bmatrix}
-1& 1
\\
1
& -1
\end{bmatrix}     ,
\\ 
\notag
&&
H_2(\zz;\la) =
\begin{bmatrix}
0 & 0
\\
-1 &-\la-1
\end{bmatrix} 
+\frac{ z_2}{z_2-z_1}
 \begin{bmatrix}
-1& 1
\\
1
& -1
\end{bmatrix}   ,
\\
\notag
&&
K(\zz,\la)
=
 \begin{bmatrix}
\frac{\la+1}{z_1\la} &\frac 1{z_1\la}
\\
\frac 1{z_2\la}&\frac{\la+1}{z_2\la}
\end{bmatrix}  .
\eean
The substitution $(\la,\mu)\, \to\, (\frac \la2, \frac12)$ transforms the system of
equations \eqref{DY}, \eqref{KZ} to the following system of equations for a column vector 
$I(\zz;\la)$, 
\bean
\label{dy}
2z_i\,\frac{\der I}{\der z_i}(\zz;\la) 
&=&
 H_i(\zz;\la) I(\zz;\la)  ,
\qquad
i=1,2,
\\
\label{kz}
I (\zz;\la+2)
&=& K(z,\la) I(\zz;\la)\, ,
\eean
Denote
\bean
\label{nabla}
&&
\mc D_i(\la)  \,:= \,2z_i\,\frac{\der }{\der z_i}- H_i(\zz;\la),
\qquad
i=1,2.
\eean

\section{Solutions modulo powers of $p$}
\label{sec 3}

\subsection{Notations}
In this paper $p$ is an odd prime integer.

\vsk.2>

In this paper we consider the system of equations \eqref{dy} and \eqref{kz} for
the values of $\la$ from  the arithmetic sequence
 of odd integers, $\La :=1+2\Z$.

\vsk.2>

Given a positive integer $s$, denote 
\bean
\La_s = \{ \la\in 1+ 2\Z \mid -p^s<\la<p^s\},
\eean
an interval of the arithmetic sequence $\La$.
\begin{itemize}
\item
For $\la\in \La_s$, we have $0<\frac{p^s-\la}2< p^s$.
\item
For a positive integer $e$, if  $s>e$, $\la\in\La_e$, then
$\big|\frac{p^{s}-\la}2\big|_p >p^{-e}$, where $|x|_p$ denotes the $p$-adic norm of a rational number $x$.

\end{itemize}
For a polynomial $f(t)$, denote by $\{f(t)\}_s$ the coefficient of
$t^{p^s-1}$ in $f(t)$.
For a function $g(\zz)$, denote by $\grad_\zz g$
the column gradient  vector $\big(\frac {\der g}{\der z_1}, \frac {\der g}{\der z_2}\big)$.

\subsection{Solutions}

For $\la\in\La_s$, define the master polynomial,
\bean
\label{Map}
\Phi_s(t;\zz;\la) :=t^{(p^s-\la)/2}(t-z_1)^{(p^s-1)/2}(t-z_2)^{(p^s-1)/2}.
\eean
Define the column vector
\bea
\Psi_s(t;\zz;\la)=\big(\Psi_{s,1}(t;\zz;\la), \Psi_{s,2}(t;\zz;\la)\big)
:=\,\Phi_s(t;\zz;\la)\Big(\frac {1}{t-z_1}, \frac {1}{t-z_2}\Big).
\eea
The coordinates of $\Psi_s$ are polynomials in $t,\zz$ with integer coefficients.
We denote 
\bean
\label{I_s}
I_s(\zz;\la)= (I_{s,1}(\zz;\la),I_{s,2}(\zz;\la)) := \{\Psi_s(t;\zz;\la)\}_s\,,
\eean
the coefficient of $t^{p^s-1}$ in $\Psi_s(t;\zz;\la)$,
 and  
\bean
\label{T_s}
 T_s(\zz;\la) := \{\Phi_s(t;\zz;\la)\}_s\,,
\eean
the coefficient of $t^{p^s-1}$ in $\Phi_s(t;\zz;\la)$.
For every $\la\in\La_s$, the functions $I_{s,1},I_{s,2},T_s$ are polynomials  in  $z_1,z_2$
with integer coefficients.

We have
\bean
\label{Ina}
\frac{1-p^s}2\, I_s(\zz;\la)\,=\, \grad_\zz T_s(\zz;\la)\,.
\eean

\begin{thm}
\label{thm ps-dy}
Let $s\in\Z_{>0}$,\,  $\la\in\La_s$,\, $i=1,2$. Then the vector $I_s(\zz;\la)$ satisfies the
congruence\,
\bean
 \label{pdy}
 \mc D_i(\la) I_s(\zz;\la) \equiv 0 \pmod{p^s}.
  \eean

\end{thm}

\begin{proof}

We have
\bean
\label{dpt}
\frac{\der \Phi_s}{\der t} = \Big(\frac{p^s-\la}{2t} + \frac{p^s-1}{2(t-z_1)}
+ \frac{p^s-1}{2(t-z_2)}\Big)\Phi_s\,.
\eean
Also
\bea
&&
-\,\frac{\der\Psi_{s,1}}{\der t}
=
\Phi_s\Big(-\frac{p^s-\la}{2} \frac1{t(t-z_1)}
-\frac{p^s -3}2 \frac1{(t-z_1)^2} -\frac {p^s-1}2\frac1{(t-z_1)(t-z_2)}\Big)
\\
&&
=\Phi_s\Big(\frac{p^s-\la}{2z_1} \Big[\frac1t-\frac1{t-z_1}\Big]
-\frac{p^s -3}2 \frac1{(t-z_1)^2} 
-\frac{p^s -1}2 \frac1{z_1-z_2}
\Big[\frac1{t-z_1}-\frac1{t-z_2}\Big]\Big).
\eea
Hence
\bea
&&
-\Phi_s\frac{p^s -3}2 \frac1{(t-z_1)^2} 
 =
-\frac{\der}{\der t}\Big(\frac{\Phi}{t-z_1}\Big)\,+
\\
&&
+\ 
\Phi_s\Big(\frac{p^s-\la}{2z_1}\Big[\frac1{t-z_1}-\frac1t\Big]
+\frac{p^s -1}2 \frac1{z_1-z_2}
\Big[\frac1{t-z_1}-\frac1{t-z_2}\Big]\Big)
\\
&&
=
-\frac{\der}{\der t}\Big(\frac{\Phi_s}{t-z_1}\Big)\,+
\Phi_s\Big(\frac{p^s-\la}{2z_1}\frac1{t-z_1}
+\frac{p^s -1}{2} \frac1{z_1-z_2}
\Big[\frac1{t-z_1}-\frac1{t-z_2}\Big]\Big)
\\
&&
+\ \Phi_s \frac 1{z_1}\Big(\frac{p^s -1}2\frac1{t-z_1}+\frac{p^s -1}2\frac 1{t-z_2}\Big)
-\frac 1{z_1} \frac{\der \Phi_s}{\der t}
\\
&&
=
-\frac{\der}{\der t}\Big(\frac{\Phi_s}{t-z_1}\Big)\,-\frac 1{z_1} \frac{\der \Phi_s}{\der t}\ 
\\
&&
+\
\Phi_s\Big(\frac{2p^s-\la-1}{2z_1}\frac1{t-z_1}+\frac{p^s -1}{2z_1} \frac 1{t-z_2}
+\frac{p^s -1}{2}  \frac1{z_1-z_2}
\Big[\frac1{t-z_1}-\frac1{t-z_2}\Big]\Big).
\eea
We have
\bean
\label{1c}
\Big\{\frac{\der}{\der t}\Big(\frac{\Phi_s}{t-z_1}\Big)\Big\}_s
\equiv 0, 
\ \ 
\Big\{\frac{\der \Phi_s}{\der t}\Big\}_s
\equiv 0, 
\ \ 
 \frac{2p^s-\la-1}{2}\equiv -\frac{\la+1}2,
\ \ \frac{2p^s-1}{2}\equiv -\frac{1}2
\eean
modulo $p^s$. Hence
\bea
2 z_1\frac{\der I_{s,1}}{\der z_1}\ \equiv\  -(\la+1) I_{s,1} -  I_{s,2}
+  \frac{z_1}{z_1-z_2} (-I_{s,1}+I_{s,2})  \pmod{p^s}\,.
\eea
We also have
\bea
\frac{\der \Psi_{s,2}}{\der z_1}
= -\frac{p^s-1}2\frac {\Phi_s}{(t-z_1)(t-z_2)}
=-\frac{p^s-1}2\frac {\Phi_s}{z_1-z_2}\Big[\frac 1{t-z_1}-\frac1{t-z_2}\Big].
\eea
Hence
\bea
2 z_1\frac{\der I_{s,2}}{\der z_1}\  \equiv \
  \frac{z_1}{z_1-z_2} (I_{s,1}-I_{s,2})  \pmod{p^s}\,.
\eea 
Equation \eqref{pdy} for $i=1$ is proved. 
Equation \eqref{pdy} for $i=2$ is proved similarly.
\end{proof}

\begin{thm}

\label{thm pqk}
Let  $s>e$ be positive integers, and $\la,\la+2\in\La_{e}$. Then the vector $I_s(\zz;\la)$ satisfies the
congruence\,:
\bean
\label{f pqk}
I (\zz;\la+2)
&\equiv & K(z,\la) I(\zz;\la)\, \pmod{p^{s-e}}.
\eean

\end{thm}

\begin{proof}
Equation \eqref{dpt} can be written as 
\bea
-\frac{\Phi_s}{t} = - \frac2{p^s-\la}\frac{\der \Phi_s}{\der t}
+ \frac{p^s-1}{p^s-\la} \Big(\frac{1}{t-z_1}
+ \frac{1}{t-z_2}\Big)\Phi_s\,.
\eea
Hence 
\bean
\label{d1}
&&
\Psi_{s,1}(\zz;\la+2) = \frac{\Phi_s(\zz;\la)}{t(t-z_1)} = 
-\frac{\Phi_s(\zz;\la)}{z_1} \Big[\frac1t-\frac1{t-z_1}\Big]
\\
\notag
&&
=
\frac{\Phi_s(\zz;\la)}{z_1(t-z_1)} 
- \frac2{(p^s-\la)z_1}\frac{\der \Phi_s}{\der t}
+ \frac{p^s-1}{p^s-\la} \Big(\frac{1}{z_1(t-z_1)}
+ \frac{1}{z_1(t-z_2)}\Big)\Phi_s(\zz;\la)
\eean
By \eqref{1c} the term $\big\{\frac2{(p^s-\la)z_1}\frac{\der \Phi_s}{\der t}\big\}_s$ is divisible at least by
$p^{s-e}$. Hence \eqref{d1} implies
\bea
I_{s,1}(\zz;\la+2) \equiv \frac{\la+1}{z_1\la} I_{s,1}(\zz;\la) +\frac{1}{z_1\la} I_{s,2}(\zz;\la) \pmod{p^{s-e}}.
\eea
Similarly we obtain
\bea
I_{s,2}(\zz;\la+2) \equiv \frac{\la+1}{z_2\la} I_{s,2}(\zz;\la) +\frac{1}{z_2\la} I_{s,1}(\zz;\la) \pmod{p^{s-e}}.
\eea
Theorem \ref{thm pqk} is proved.
\end{proof}

\subsection{Formulas for $I_{s,1},I_{s,2},T_s$ }

\begin{lem}

For $\la\in\La_s$  we have
\bean
\label{fT_s}
T_s(\zz;\la)
&=& (-1)^{\frac{p^s-\la}2}
\sum_{k+\ell =\frac{p^s-\la}2} \binom{\frac{p^s-1}2}{k}\binom{\frac{p^s-1}2}{\ell}\,z_1^kz_2^\ell\,,
\\
\label{fI_{s,1}}
I_{s,1}(\zz;\la)&=& (-1)^{\frac{p^s-\la}2-1}
\sum_{k+\ell =\frac{p^s-\la}2-1} \binom{\frac{p^s-1}2-1}{k}\binom{\frac{p^s-1}2}{\ell}\,z_1^kz_2^\ell\,,
\\
\label{fI_{s,2}}
\pushQED{\qed} 
I_{s,2}(\zz;\la)&=&(-1)^{\frac{p^s-\la}2-1}
\sum_{k+\ell =\frac{p^s-\la}2-1} \binom{\frac{p^s-1}2}{k}\binom{\frac{p^s-1}2-1}{\ell}\,z_1^kz_2^\ell\,.\qquad
\qedhere
\popQED
\eean

\end{lem}

\subsection{$p$-ary representations}

For $\la\in\La_s$  we have the following $p$-ary representations
\bean
\label{vws}
\frac{p^s-\la}2 &=& w_0(\la)+w_1(\la)p+\dots +w_{s-1}(\la)p^{s-1}, 
\\
\label{-la}
\frac{-\la}2 &=& w_0(\la)+w_1(\la)p+\dots +w_{s-1}(\la)p^{s-1} +\frac{p-1}2p^s + \frac{p-1}2p^{s+1} + \dots
\eean 
for some integers  $w_i(\la)$, \, $0\leq w_i(\la)\leq p-1$. Denote $w_i(\la) = \frac{p-1}2$ for $i\geq s$.

\vsk.2>
For $\la\in\La$, denote by $W(\la)$ the set of all distinct integers $w_i(\la)$ in the $p$-ary representation of 
$\frac{-\la}2$. The set $W(\la)$ has at most  $p$ elements.

\vsk.2>
For example, $\frac{-1}2 = \frac{p-1}2(1+p+\dots)$ and $W(1)=\{\frac{p-1}2\}$,
while $\frac{1}2 = \frac{p+1}2+\frac{p-1}2(p+p^2+\dots)$, and $W(-1)=\{\frac{p+1}2, \frac{p-1}2\}$.

\vsk.2>

For $w=0, 1, \dots, p-1$,\, let $h(\zz;w)$ (resp.  $g_1(\zz;w)$, resp. $g_2(\zz;w)$)
be the coefficient of $t^{p-1}$ in  $t^w(t-z_1)^{(p-1)/2} (t-z_2)^{(p-1)/2}$
(resp. in $t^w(t-z_1)^{(p-1)/2-1} (t-z_2)^{(p-1)/2}$, resp. in 
$t^w(t-z_1)^{(p-1)/2} (t-z_2)^{(p-1)/2-1}$).

\begin{lem}
\label{lem 3.4}
For $\la\in\La_s$  we have
\bean
\label{Tsh}
T_s(\zz;\la) \equiv \prod_{i=0}^{s-1} h\big(\zz^{p^i};\,w_i(\la)\big) \pmod{p}  .
\eean
The polynomial  $T_s(\zz;\la)$ is nonzero modulo $p$.

\end{lem}

\begin{proof}

We have $\frac{p^s-1}2=\frac{p-1}2\,(1+p+\dots+p^{s-1})$.
Then
\bean
\label{mop}
\Phi_s(t;\zz;\la) 
&\equiv&
 \prod_{i=0}^{s-1}
 (t^{p^i})^{w_i(\la)}(t^{p^i}-z_1^{p^i})^{(p-1)/2}(t^{p^i}-z_2^{p^i})^{(p-1)/2} \pmod{p}.
\eean
This implies \eqref{Tsh}.
To prove the second statement of the lemma it is enough to check that 
the polynomial $h(\zz;w)$  is nonzero modulo $p$ for $w=0,1,\dots,p-1$.
Indeed, there exist nonnegative integers 
$k,\ell$ such that $w=k+\ell$ and 
$k,\ell\leq (p-1)/2$. Then
the coefficient  of $z_1^kz_2^\ell$ in $h(\zz;w)$ equals
$(-1)^w\binom{\frac{p-1}2}{k}\binom{\frac{p-1}2}{\ell}$ and 
is nonzero modulo $p$ by Lucas Theorem.
\end{proof}

\begin{lem}
\label{lem 3.5}
Let $\la\in \La_s$ and $j=1,2$. Then
\bean
\label{Tsh1}
I_{s,j}(\zz;\la) \equiv g_j(\zz;w_0(\la)) \prod_{i=1}^{s-1} h\big(\zz^{p^i};w_i(\la)\big) \pmod{p}  .
\eean
If $\la$ is not divisible by $p$, then $I_{s,j}(\zz;\la)$ is nonzero modulo $p$.

\end{lem}

\begin{proof} 
If $\la$ is not divisible by $p$, then $w_0(\la)>0$. Then $g_j(\zz;w_0(\la))$ is nonzero modulo $p$.
Also,  the polynomials $h(\zz;w_i(\la))$ are nonzero modulo $p$.
\end{proof}

\section{Dwork-type congruences}
\label{sec 4}

In this section we apply results from \cite{V5} to 
obtain congruences relating the functions $I_s(\zz;\la)$, $T_s(\zz;\la)$ in $\zz$\,\ for different
$s$. This type of congruences was originated by B.\,Dwork in \cite{Dw}, see also, for example,
\cite{Me, MeV, VZ1, VZ2}.

\medskip

A congruence $F(x)\equiv G(x)\pmod{p^s}$ for two polynomials in some variables $x$ with integer coefficients
 is understood 
as the divisibility by $p^s$ of all coefficients of $F(x)-G(x)$.

\vsk.2>

Let $F_1(x)$, $F_2(x)$, $G_1(x)$, $G_2(x)$ be polynomials 
such that  $F_2(x)$, $G_2(x)$  are  both nonzero  modulo~$p$.
Then the congruence
$F_1(x)/F_2(x)\equiv G_1(x)/G_2(x)$ modulo $p^s$
is understood as the congruence
\bea
F_1(x)G_2(x)\equiv G_1(x)F_2(x) \pmod{p^s}\,.
\eea

\vsk.2>

\medskip

Recall the master polynomial
\bea
\Phi_s(t;\zz;\la)=t^{(p^s-\la)/2}(t-z_1)^{(p^s-1)/2}(t-z_2)^{(p^s-1)/2},
\eea
in particular,
\bea
\Phi_1(t;\zz;1)=t^{(p-1)/2}(t-z_1)^{(p-1)/2}(t-z_2)^{(p-1)/2}.
\eea
For $\la\in \La_e$ and $s>e$ we have
\bean
\label{rl}
\Phi_s(t;\zz;\la) = \Phi_e(t;\zz;\la)\Phi_1(t;\zz;1)^{p^{e}+p^{e+1}+\dots+p^{s-e-1}},
\eean
In particular, we have
\bea
\Phi_{s-e}(t;\zz;1) = \Phi_1(t;\zz;1)^{1+p+\dots +p^{s-e-1}}\,.
\eea

\noindent
For $\la\in\La_e$ and $s\geq e$,
 the Newton polytope of $\Phi_s(t;\zz;\la)$ with respect to the variable $t$ is the interval
$\big[\frac{p^s-\la}2, \,\frac{p^s-\la}2+p^s-1\big]$. 
\bean
\label{NPc}
\on{For\, a\, positive\,integer}\, k, \,\on{ the \,point}\,kp^s-1\,
\on{lies\, in\, this\, interval\, only\, if}\,  k=1.
\eean

\vsk.2>
Recall that 
$T_s(\zz;\la) = \{\Phi_s(t;\zz;\la)\}_s$\, is
the coefficient of $t^{p^s-1}$ in $\Psi_s(t;\zz;\la)$.
\bean
\label{nonz}
\on{For}\, \la\in\La_e,\, \on{the\,polynomial}\,
T_s(\zz;\la) \,
\on{is\, nonzero\,modulo}\,p\,,
\eean
by Lemma \ref{lem 5.4}.

\vsk.2>
In \cite{V5}, certain congruences were proved for a sequence of polynomials 
 like $\Phi_s(\zz;\la)$, $s\geq e$, with  properties like \eqref{rl}, \eqref{NPc},  \eqref{nonz}.
In the case of the polynomials $\Phi_s(\zz;\la)$, $s\geq e$, the congruences in \cite{V5}  say the following.

\begin{thm}
\label{thm 1.6}

Let  $e\in\Z_{>0}$,\, $\la\in\La_e$.
 
\begin{enumerate}


 
 \item[\textup{(i)}] 
For $j\in \{1,2\}$ denote $D_j = \frac\der {\der z_j}$. Then for
$s>e$ we have
\bean
\label{Der}
\frac{D_j(T_s(\zz;\la))}{ T_{s}(\zz;\la)}
\equiv
\frac{D_j(T_{s-1}(\zz;\la))}{ T_{s-1}(\zz;\la)}
\pmod{p^{s-e}}.
\eean

 \item[\textup{(ii)}] 
For $i,j\in \{1,2\}$ and 
$s>e$ we have
\bean
\label{DDer}
\frac{D_i(D_j(T_s(\zz;\la)))}{ T_{s}(\zz;\la)}
\equiv
\frac{D_i(D_j(T_{s-1}(\zz;\la)))}{ T_{s-1}(\zz;\la)}
\pmod{p^{s-e}}.
\eean

\end{enumerate}

\end{thm}

Statements (i-ii) are special cases of \cite[Theorems 2.8 and 2.9]{V5}.

\begin{cor}
\label{cor TI} 
Let $e\in\Z_{>0}$,\, $\la \in \La_e$,\, $s>e$, and $i,j\in \{1,2\}$. Then
\bean
\label{TI1}
\frac{I_{s,j}(\zz;\la)}{ T_{s}(\zz;\la)}
&\equiv&
\frac{I_{s-1,j}(\zz;\la)}{ T_{s-1}(\zz;\la)}
\pmod{p^{s-e}}\,,
\\
\label{TI2}
\frac{\frac{\der I_{s,j}}{\der z_i}(\zz;\la)}{ T_{s}(\zz;\la)}
&\equiv&
\frac{\frac{\der I_{s-1,j}}{\der z_i}(\zz;\la)}{ T_{s-1}(\zz;\la)}
\pmod{p^{s-e}}\,.
\eean

\end{cor}

The corollary follows from formula \eqref{Ina} and Theorem \ref{thm 1.6}.

\begin{thm}
\label{thm 6.3}

Let $e\in\Z_{>0}$. Let  $\la, \la+2\in\La_e$ and $s>2e$.
Then
\bean
\label{6.9}
\frac {I_s(\zz;\la+2)}{T_s(\zz;\la)}  \equiv 
\frac {I_{s-1}(\zz;\la+2)}{T_{s-1}(\zz;\la)} 
\pmod{p^{s-2e}}\,.
\eean

\end{thm}

\begin{proof}
Using formulas \eqref{TI1}, \eqref{TI2} and the fact that $\la\in\La_e$, we obtain
\bea
\phantom{aa}
&
\frac1{\la} \begin{bmatrix}
\frac{\la+1}{z_1} &\frac 1{z_1}
\\
\frac 1{z_2}&\frac{\la+1}{z_2}
\end{bmatrix} \frac {I_{s}(\zz;\la)}{T_{s}(\zz;\la)}\, 
\equiv
\frac1{\la} \begin{bmatrix}
\frac{\la+1}{z_1} &\frac 1{z_1}
\\
\frac 1{z_2}&\frac{\la+1}{z_2}
\end{bmatrix} \frac { I_{s-1}(\zz;\la)}{T_{s-1}(\zz;\la)}\,\, \pmod{p^{s-2e}}.
\eea
Congruence \eqref{f pqk} implies the congruences:
\bea
 \frac {I_s (\zz;\la+2)}{T_s(\zz;\la)}\,
&\equiv&
\frac1{\la} \begin{bmatrix}
\frac{\la+1}{z_1} &\frac 1{z_1}
\\
\frac 1{z_2}&\frac{\la+1}{z_2}
\end{bmatrix} \frac {I_s(\zz;\la)}{T_s(\zz;\la)}\, \pmod{p^{s-e}},
\\
\phantom{aa}
 \frac {I_{s-1} (\zz;\la+2)}{T_{s-1}(\zz;\la)}
&\equiv&
\frac1{\la} \begin{bmatrix}
\frac{\la+1}{z_1} &\frac 1{z_1}
\\
\frac 1{z_2}&\frac{\la+1}{z_2}
\end{bmatrix} \frac {I_{s-1}(\zz;\la)}{T_{s-1}(\zz;\la)}\, \pmod{p^{s-e-1}}.
\eea
These three congruences  imply congruence \eqref{6.9}.
\end{proof}

\section{Convergence}
\label{sec 5}

\subsection{Unramified extensions of $\Q_p$}

We fix  an algebraic closure $\overline{\Q_p}$ of $\Q_p$.
For every $m$, there is a unique unramified extension of $\Q_p$ in 
 $\overline{\Q_p}$ of degree $m$, denoted by $\Q_p^{(m)}$.
This can be obtained by attaching to $\Q_p$ a primitive root of $1$ of order $p^m-1$.
The norm $|\cdot|_p$ on $\Q_p$ extends to a norm $|\cdot|_p$ on 
$\Q_p^{(m)}$.
Let 
\bea
\Z_p^{(m)} = \{ a\in \Q_p^{(m)} \mid |a|_p\leq 1\}
\eea
denote the ring of integers in $\Q_p^{(m)}$. The ring $\Z_p^{(m)}$
has the unique maximal ideal 
\bea
\mathbb M_p^{(m)} = \{ a\in \Q_p^{(m)} \mid |a|_p <1\},
\eea
such that $\mathbb Z_p^{(m)}\big/ \mathbb M_p^{(m)}$ is isomorphic to the finite field
$\F_{p^m}$.

For every $t\in\F_{p^m}$ there is a unique $\tilde t\in \mathbb Z_p^{(m)}$ that is a lift of $t$ and such that 
$\tilde t^{p^m}=\tilde t$. The element $\tilde t$ is called the Teichmuller lift of $t$.

\subsection{Domain $\frak D_B$}
For $t\in\F_{p^m}$ and $r>0$ denote
\bea
D_{t,r} = \{ a\in \Z_p^{(m)}\mid |a-\tilde t|_p<r\}\,.
\eea
We have the partition
\bea
\Z_p^{(m)} = \bigcup_{t\in\F_{p^m}} D_{t,1}\,.
\eea
Recall  $\zz=(z_1,z_2)$. For $B(\zz) \in \Z[\zz]$, define
\bea
\frak D_B \ =\ \{ \aa\in (\Z_p^{(m)})^2\,  \mid \  |B(\aa)|_p=1\} .
\eea
Let $\bar B(\zz)$ be the projection of $B(\zz)$ to $\F_p[\zz]\subset \F_{p^m}[\zz]$.
Then $\frak D_B$ is the union of unit polydiscs,
\bea
\frak D_B = \bigcup_{\substack{t_1,t_2\in \F_{p^m}\\  \bar B(t_1,t_2)\ne 0}} \  D_{t_1,1}\times D_{t_2,1}\,.
\eea


\begin{lem}
\label{lem 7.1}

For any nonnegative integer $k$ we have
\begin{align}
\notag
\{ \aa\in (\Z_p^{(m)})^2 \mid \ |B(\aa^{p^k})|_p=1\}
&=\bigcup_{\substack{t_1,t_2\in \F_{p^m}\\  \bar B(t_1^{p^k},t_2^{p^k})\ne 0}} \  D_{t_1,1}\times  D_{t_2,1} =
\\
\notag
&=\bigcup_{\substack{t_1,t_2\in F_{p^m}\\  \bar B(t_1,t_2)\ne 0}} \  
D_{t_1,1}\times D_{t_2,1} =
 \frak D_B \,.
\end{align}
\qed
\end{lem}

\begin{lem}
[{\cite[Lemma 6.1]{VZ2}}]
\label{lem nonempty}

Let $\bar B(\zz) \in  \F_p[\zz]$ be a nonzero polynomial  of degree $d$,
and $d+1< p^m$. Then the set
$ \{ \aa\in (\F_{p^m})^2\mid \bar B(\aa) \ne 0\}$
is nonempty. Moreover, there are at least
$\frac{p^{2m}-1}{p^m-1} (p^m-1-d) + 1 $ points
of\  $(\F_{p^m})^2$ where $\bar B(\zz)$ is nonzero.

\end{lem}

\subsection{Domains of convergence}

Recall the polynomials $T_s(\zz;\la)$,
$I_{s,1}(\zz;\la)$, $I_{s,2}(\zz;\la)$ as well as the polynomials
$h(\zz;w)$, $g_{1}(\zz;w)$, $g_{2}(\zz;w)$. For $\la\in \La$, denote
\bea
H(\zz;\la) 
&=&
\prod_{w\in W(\la)} h(\zz;w),
\\
\frak D^{(m)}(\la)
&=&
 \{\aa \in (\Z_p^{(m)})^{2}  \mid  |H(\aa;\la)|_p=1\}.
\eea 
Let $\la\in\La$ be not divisible by $p$ (that is,\, $w_0(\la)>0$).
For $j=1,2$, denote
\bea
G_j(\zz;\la) 
&=&
g_j(\zz;w_0(\la))\prod_{w\in W(\la)} h(\zz;w),
\\
\frak D^{(m)}_*(\la)
&=& \{\aa \in (\Z_p^{(m)})^{2}  \mid  |G_1(\aa;\la)|_p=1\ \on{or}\ |G_2(\aa;\la)|_p=1\}.
\eea 
Denote $\frak C^{(m)} =  \{\aa \in (\Z_p^{(m)})^{2}  \mid  a_1a_2\ne 0\}$.
For $\la\in\La$ divisible by $p$, denote
\bea
\frak D^{(m)}_*(\la) 
= \frak D^{(m)}(\la) \cap \frak D^{(m)}_*(\la+2) \cap \frak C^{(m)}.
\eea 
Clearly, for any $\la\in\La$, we have
\bea
\frak D^{(m)}_*(\la)  \subset \frak D^{(m)}(\la), 
\eea
and
\bean
\label{inter}
\phantom{aaaa}
\bigcap\nolimits_{\la\in\La} \frak D^{(m)}_*(\la) \
 \mbox{\large$\supset$} \
\big\{\aa \in (\Z_p^{(m)})^{2}  \mid  \big| a_1a_2h(\aa;0)\prod_{w=1}^{p-1} h(\aa;w)g_1(\aa;w)g_2(\aa;w)\big|_p=1\big\}.
 \eean

\begin{lem}

If $m\geq 3$, then\ $\bigcap_{\la\in\La} \frak D^{(m)}_*(\la)$\
 is nonempty.

\end{lem}

\begin{proof}
We have $\deg_\zz h(\zz;w) = w$ and 
$\deg_\zz g_1(\zz;w) = \deg_\zz g_2(\zz;w) = w-1$.
Hence the polynomial in \eqref{inter} has degree 
$ \frac{3p^2-7p+8}2<p^3-1$. The lemma  follows from 
Lemma \ref{lem nonempty}. 
\end{proof}

\begin{thm}
\label{thm coKZ}
 For $e\in\Z_{>0}$ and $\la\in\La_e$, we have the following statements.

\begin{enumerate}
\item[$(\on{i})$]

The  sequence of column vectors,
$\Big(\frac {I_{s}(\zz;\la)}{T_{s}(\zz;\la)} \Big)_{s\geq e}$\,\,,
whose entries are rational functions in $\zz$ regular on  $\frak D^{(m)}(\la)$,
uniformly converges on $\frak D^{(m)}(\la)$ as $s\to\infty$ to a
a vector, whose entries are analytic functions on $\frak D^{(m)}(\la)$.
The vector will  be denoted by $\mc I(\zz;\la) = (\mc I_1(\zz;\la), \mc I_2(\zz;\la))$.

\item[$(\on{ii})$]

The  sequence of column vectors
$\Big(\frac {I_s(\zz;\la+2)}{T_s(\zz;\la)} \Big)_{s> 2e}$\,,
whose entries are rational functions in $\zz$ regular on  $\frak D^{(m)}(\la)$,
uniformly converges on $\frak D^{(m)}(\la)$ as $s\to\infty$ to 
a vector, whose entries are analytic functions on $\frak D^{(m)}(\la)$.
The vector will  be denoted by
\\
 $\tilde{\mc I}(\zz;\la+2) = (\tilde{\mc I}_1(\zz;\la+2), \tilde{\mc I}_2(\zz;\la+2))$.

\item[$(\on{iii})$]
For  $j=1,2$,  the sequence of column vectors,
$\Big(\frac {\frac{\der I_{s}}{\der z_j}(\zz;\la)}{T_{s}(\zz;\la)} \Big)_{s\geq e}$\,,
whose entries are rational functions in $\zz$ regular on  $\frak D^{(m)}(\la)$,
uniformly converges on $\frak D^{(m)}(\la)$ as $s\to\infty$ to 
a vector, whose entries are analytic functions on $\frak D^{(m)}(\la)$.
The vector will  be denoted by $\mc I^{(i)}(\zz;\la) = (\mc I_1^{(i)}(\zz;\la), \mc I_2^{(i)}(\zz;\la))$.

\end{enumerate}

\end{thm}

\begin{proof} 
Parts (i), (ii), (iii) follow from the congruences of \eqref{TI1}, \eqref{TI2}, \eqref{6.9}, respectively.
\end{proof}

\subsection{Relations between limiting vectors}

\begin{lem}
\label{lem 7.6}

Let $\la\in\La$.
We have the following equations on $\frak D^{(m)}(\la)$\,:
\bean
\label{II}
\frac{\der \mc I}{\der z_i}(\zz;\la) 
&=&
 \mc I^{(i)}(\zz;\la) - \frac12  \mc I_i(\zz;\la) \mc I (\zz;\la)\,,
\\
\label{7.4}
\mc I^{(i)}(\zz;\la)
& =& 
H_i(\zz;\la)\mc I(\zz;\la),
\\
\label{7.5}
\frac{\der \mc I}{\der z_i}(\zz;\la) &=&
 \Big(H_i(\zz;\la)  - \frac 12 \mc I_i(\zz;\la)\Big) \mc I(\zz;\la),
\\
\label{7.6}
\tilde {\mc I}(\zz;\la+2) &=& K(\zz;\la)\mc I(\zz;\la).
\eean

\end{lem}

\vsk.2>

\begin{proof}  
Differentiating the congruence in \eqref{TI1} with respect to $z_i$ we obtain
\bea
\frac{\frac{\der I_{s,j}}{\der z_i}}{ T_{s}}
- \frac{\frac{\der T_s}{\der z_i}}{ T_{s}}
\frac{I_{s,j}}{ T_{s}}
&\equiv&
\frac{\frac{\der I_{s-1,j}}{\der z_i}}{ T_{s-1}}
- \frac{\frac{\der T_{s-1}}{\der z_i}}{ T_{s-1}}
\frac{I_{s-1,j}}{ T_{s-1}}
\pmod{p^{s-e}}\,.
\eea
This congruence gives equation \eqref{II} as $s\to \infty$.
Equation \eqref{7.4}  follows from the congruences in \eqref{pdy}
after dividing by $T_s(\zz;\la)$ and taking the limit $s\to\infty$.
Equation \eqref{7.5} follows from equations \eqref{II} and \eqref{7.4}.

Equation \eqref{7.6} follows from  congruence \eqref{f pqk} after dividing by 
$T_s(\zz;\la)$ and taking the limit $s\to\infty$.
\end{proof}

\subsection{Limiting vectors}


The vector-function 
$\tilde {\mc I}(\zz;\la+2) =(\tilde {\mc I}_1(\zz;\la+2),\tilde {\mc I}_2(\zz;\la+2))$  
is defined on $\frak D^{(m)}(\la)$, and the vector-function 
${\mc I}(\zz;\la+2) =( {\mc I}_1(\zz;\la+2), {\mc I}_2(\zz;\la+2))$  
is defined on $\frak D^{(m)}(\la+2)$, see Theorem \ref{thm coKZ}.

\begin{lem}
\label{lem pr}

The vector-functions 
$\tilde {\mc I}(\zz;\la+2)$ and ${\mc I}(\zz;\la+2)$ are proportional on
\linebreak
$\frak D^{(m)}(\la)\cap \frak D^{(m)}(\la+2)$, that is,
\bean
\label{7.7}
\tilde {\mc I}_1(\zz;\la+2) {\mc I}_2(\zz;\la+2)-
\tilde {\mc I}_2(\zz;\la+2) {\mc I}_1(\zz;\la+2) = 0.
\eean

\end{lem}

\begin{proof}
To obtain  $\tilde {\mc I}(\zz;\la+2)$ we divide  the vector t $I_s(\zz;\la+2)$ by $T_s(\zz;\la)$ and take the limit
as $s\to\infty$.  To obtain  $ {\mc I}(\zz;\la+2)$ we divide  the
same vector $I_s(\zz;\la+2)$ by $T_s(\zz;\la+2)$ and take the limit
as $s\to\infty$. Hence the limits are proportional.
\end{proof}

\begin{lem}
\label{lem 5.4}

Let $e\in\Z_{>0}$,\, $\la\in\La_e$. 

\begin{enumerate}
\item[$\on{(i)}$]

Assume that $\aa \in  \frak D^{(m)}(\la)$.  Then  $|T_s(\aa;\la)|_p=1$
for any $s\geq e$.

\item[$\on{(ii)}$]
Assume that  $\la$ is not divisible by $p$ and $\aa \in  \frak D^{(m)}_*(\la)$. Then there exists $j\in\{1,2\}$ such that
 $|I_{s,j}(\aa;\la)|_p=1$ for any $s\geq e$.

\end{enumerate}
\end{lem}

\begin{proof} The lemma  follows from Lemmas \ref{lem 3.4} and \ref{lem 3.5}.
\end{proof}

\begin{lem}
\label{lem 7.7}
${}$

\begin{enumerate}
\item[$(\on{i})$]
Let $\la\in\La$ be  not divisible by $p$ and
$\aa\in \frak D^{(m)}_*(\la)$.  Then there exists $j\in\{1,2\}$ such that 
 $|\mc I_j(\aa;\la)|_p =1$.

\item[$(\on{ii})$]

Let $\la+2\in\La$ be not divisible by $p$ and
 $\aa\in \frak D^{(m)}(\la) \cap
 \frak D^{(m)}_*(\la+2)$. Then there exists $j\in\{1,2\}$ such that
  $|\tilde{\mc I}_j(\aa;\la+2)|_p =1$.

\item[$(\on{iii})$]

Let  $\la\in \La$  be  divisible by $p$ and  $\aa\in \frak D^{(m)}_*(\la)$.
Then the vector $\mc I(\aa;\la)$ is nonzero.

\end{enumerate}

\end{lem}

\begin{proof}
Under assumptions of part (i), there exists $j\in\{1,2\}$ such that 
$ |G_j(\aa;\la)|_p=1$. Then
$|I_{s,j}(\aa;\la)|_p = |T_{s}(\aa;\la)|_p =1$ for all large $s$
by Lemmas \ref{lem 3.4}, \ref{lem 3.5}, \ref{lem 7.1}, \ref{lem 5.4}.  Part (i) is proved.

Under assumptions of part (ii), there exists $j\in\{1,2\}$ such that 
$ |G_j(\aa;\la+2)|_p=1$. Then
$|I_{s,j}(\aa;\la+2)|_p = |T_{s}(\aa;\la)|_p =1$ for all large $s$
by Lemmas \ref{lem 3.4}, \ref{lem 3.5}, \ref{lem 7.1}, \ref{lem 5.4}.  Part (ii) is proved.

To prove part (iii) consider equation \eqref{7.6},
\bean
\label{tII}
\tilde {\mc I}(\aa;\la+2) = \frac1{\la} 
\begin{bmatrix}
\frac{\la+1}{a_1} &\frac 1{a_1}
\\
\frac 1{a_2}&\frac{\la+1}{a_2}
\end{bmatrix} \mc I(\aa,\la),
\eean
which holds for $\aa\in\frak D^{(m)}(\la)$.
By part (ii), the vector  $\tilde{\mc I}(\aa;\la+2)$ is nonzero 
 for $\aa\in\frak D^{(m)}(\la)\cap
\frak D^{(m)}_*(\la+2)$.
Since $\aa \in \frak D^{(m)}(\la) \cap \frak D^{(m)}_*(\la+2) \cap \frak C^{(m)}$
we have  $a_1a_2\ne 0$. Hence the matrix in \eqref{tII} is well defined, and therefore
 $\mc I(\aa;\la)$ is nonzero. 
\end{proof}

\subsection{Invariant line bundle} 

Denote  $\mc W=(\Q_p^{(m)})^2$,
The differential operators 
\bea
\mc D_i(\la) = \frac{\der}{\der z_i}-H_i(\zz;\la), \qquad i=1,2,
\eea 
define a connection   on 
the trivial bundle $\mc W\times (\Z_p^{(m)})^2 \times \La \to (\Z_p^{(m)})^2 \times \La$
called the dynamical connection. 


For any $\la\in\La$, we have a map of local sections of the bundle
 $\mc W\times (\Z_p^{(m)})^2\times \{\la\} \to (\Z_p^{(m)})^2\times \{\la\}$
to local sections of the bundle $\mc W\times (\Z_p^{(m)})^2\times \{\la+2\} 
\to (\Z_p^{(m)})^2\times \{\la+2\}$
defined by the formula,
\bean
\label{tau}
\tau \ :\  s(\zz) \ \mapsto \  K(\zz;\la) s(\zz).
\eean
We call the operator $\tau$ the \qKZ/ discrete connection on
 the trivial bundle 
$\mc W\times (\Z_p^{(m)})^2\times \La \to (\Z_p^{(m)})^2 \times \La$.

The dynamical and \qKZ/ connections are compatible. Namely, for $\la\in\La$ and a local section $s(\zz)$
of  $\mc W\times (\Z_p^{(m)})^2\times \{\la\} \to (\Z_p^{(m)})^2\times \{\la\}$ we have
\bea
\mc D_1(\la)\big(\mc D_2(\la) s(\zz)\big) 
&=&
\mc D_2(\la)\big(\mc D_1(\la) s(\zz)\big),
\\
\tau\big(\mc D_i(\la) s(\zz)\big)
&=&
\mc D_i(\la+2) \big(\tau s(\zz)\big),
\qquad i=1,2.
\eea

\vsk.2>

Denote 
\bea
\frak D^{(m)}[\La] 
&:=&
 \bigcup\nolimits_{\la\in\La} \frak D^{(m)}_*(\la) \times \{\la\} \ \subset \ (\Z_p^{(m)})^2\times \La.
\eea
For any $(\aa,\la)\in \frak D^{(m)}[\La]$,\ the vector 
$\mc I(\aa,\la)$ is nonzero by Lemma \ref{lem 7.7}.

\vsk.2>
For any $(\aa,\la)\in \frak D^{(m)}[\La]$,\
let $\mc L_{(\aa,\la)} \subset \mc W$ 
be the one-dimensional vector subspace generated by 
the  vector $\mc I(\aa,\la)$.   Then
\bea
\mc L := \bigcup\nolimits_{(\aa,\la)\in  
\frak D^{(m)}[\La]}\,\mc L_{(\aa,\la)}\times \{(\aa,\la)\} 
\,\to\, \frak D^{(m)}[\La]
\eea
is an analytic line subbundle of  the trivial bundle
$\mc W\times \frak D^{(m)}[\La]
 \to \frak D^{(m)}[\La]$

\begin{thm}
\label{thm inv} 

The line bundle $\mc L  \,\to\,  \frak D^{(m)}[\La]$
  is invariant with respect to the dynamical and  \qKZ/ connections.
More precisely,
\begin{enumerate}
\item[$\on{(i)}$]
if $s(\zz)$ is a local section of $\mc L $ over $\frak D^{(m)}_*(\la)\times \{\la\}$, then 
$\mc D_i(\la) s(\zz)$, $i=1,2$,  also are  local sections of $\mc L$ over 
 $\frak D^{(m)}_*(\la)\times \{\la\}$;

\item[$\on{(ii)}$]

if $s(\zz)$ is a local section of $\mc L$ over 
 $(\frak D^{(m)}_*(\la)\cap \frak D^{(m)}_*(\la+2))\times \{\la\}$, then 
 $\tau s(\zz)$,  is a local section of $\mc L$ over 
  $(\frak D^{(m)}_*(\la)\cap \frak D^{(m)}_*(\la+2)) \times \{\la+2\}$.

\end{enumerate}
\end{thm}

\begin{proof} 

Let $(\aa,\la)\in  \frak D^{(m)}_*(\la)\times \{\la\}$. 
 Let $c(\zz)$ be a scalar analytic function at $\aa$.
Consider the local section  $c(\zz) \mc I(\zz;\la)$ of $\mc L$ at  $(\aa,\la)$. Then
\bea
\mc D_i(\la) \big(c(\zz) \mc I(\zz;\la)\big)
&=&
 -\, cH_i \mc I  +c\frac{\der \mc I}{\der z_i}  + \frac{\der c}{\der z_i}\mc I
\\
&=&
 - \,c H_i \mc I +c\Big( \mc I^{(i)} - \frac12 \mc I_i \mc I\Big) 
+   \frac{\der c}{\der z_i}\mc I
\\
&=&
 - \,cH_i \mc I  +c \Big( H_i  \mc I  - \frac12 \mc I_i \mc I\Big) 
+  \frac{\der c}{\der z_i}\mc I
\\
&=&
  \Big(-  \frac{c}2 \mc I_i  
+  \frac{\der c}{\der z_i}\Big)\mc I \,.
\eea
Here we used Lemma \ref{lem 7.6}.
Clearly, the last expression is a local section of  $\mc L$ at $(\aa,\la)$.
Part (i) is proved.

By definition of $\tau$, we have  $\tau \big(c(\zz) \mc I(\zz;\la)\big) = c(\zz)K(\zz;\la)\mc I(\zz;\la)$.
  We also have the equality
\bea
c(\zz)K(\zz;\la)  \mc I(\zz;\la) = c(\zz) \tilde{\mc I}(\zz;\la+2),
\eea
which holds on  $\frak D^{(m)}(\la)$,
by Lemma \ref{lem 7.6}.   The vectors
$\tilde {\mc I}(\zz;\la+2)$  
and
${\mc I}(\zz;\la+2)$
  are proportional on $\frak D^{(m)}(\la)\cap \frak D^{(m)}(\la+2)$, by Lemma \ref{lem pr}.
For the smaller set $\frak D^{(m)}_*(\la)\cap \frak D^{(m)}_*(\la+2)$,
the initial vector vector  $ {\mc I}(\zz;\la)$ and the resulting vector  ${\mc I}(\zz;\la+2)$ are
both nonzero, by Lemmas \ref{lem 5.4} and \ref{lem 7.7}. This proves part (ii).
\end{proof}

\subsection{Special points}

\begin{lem}
\label{lem sp}
The points $\big(0,1;1\big)$, $\big(1,0;1\big)$,
$\big(1,1;1\big)$ belong to
$\frak D^{(m)}_*(1) \times \{1\}\subset \frak D^{(m)}[\La]$.
\end{lem}

\begin{proof}
Straightforward calculation shows that
$T_s\big(0,1;1\big)=(-1)^{(p^s-1)/2}$,
$I_{s,1}\big(0,1;1\big)=$
\\ $(-1)^{(p^s-3)/2}$,
$I_{s,2}\big(0,1;1\big)=(-1)^{(p^s-3)/2}\frac{p^s-1}2$.
Hence $\mc I\big(0,1;1\big) = \big(-1,\frac12\big)$.
Similarly we obtain $\mc I\big(1,0;1\big) = \big(\frac12,-1\big)$,
$\mc I\big(1,1;1\big) = \big(-\frac12,-\frac12\big)$.
These vectors are nonzero modulo $p$.
\end{proof}

By Lemma \ref{lem sp}, the analytic vector-function $\mc I(\zz;1)$ is nonzero at 
$(0,1;1)$, and its values generate the line bundle $\mc L$  over a neighborhood of $(0,1)$ in 
$\frak D^{(m)}_*(1)$. Over a neighborhood of $(0,1)$, 
the   same line bundle can be defined differently.  Consider the family of
 elliptic curves $X(\zz)$  defined by  the equation
$y^2 =t(t-z_1)(t-z_2)$. If the parameter $(z_1,z_2)$ is close to $(0,1)$ the curve
$X(\zz)$ has a vanishing cycle denoted by $C_{0,1}$. The vector-function
\bean
\label{Ic}
I^{(C_{0,1})}(\zz;1)\ :=\ \int_{C_{0,1}}  \Big(\frac {1}{(t-z_1)y}\,,\ \frac {1}{(t-z_2)y}\Big)dt 
\eean
is holomorphic at $(0,1)$, solves dynamical equations \eqref{dy} 
for $\la=1$,  and $I^{(C)}(0,1;1)\ne 0$. 
The values of $I^{(C)}\big(\zz;1\big)$ generate a line bundle denoted by
$\mc L_{0,1}$ over a neighborhood of $(0,1)$. The line bundle $\mc L_{0,1}$ is invariant
with respect to the dynamical connection. The dynamical connection for $\la=1$
does not have other invariant proper nontrivial subbundles near $(0,1)$ since other
solutions of equations  \eqref{dy} for $\la=1$ at (0,1) include $\log z_1$. 
Hence our line bundle $\mc L$ coincides with the line bundle
$\mc L_{0,1}$ over a neighborhood of $(0,1)\subset \frak D^{(m)}_*(1)$.

\vsk.2>

Similarly, the elliptic curve $X(\zz)$  has a vanishing cycle denoted by $C_{1,0}$
if the parameter $(z_1,z_2)$ is close to $(1,0)$. The values of the nonzero vector-function
\bea
I^{(C_{1,0})}(\zz;1)\ :=\ \int_{C_{1,0}}  \Big(\frac {1}{(t-z_1)y}\,,\ \frac {1}{(t-z_2)y}\Big)dt 
\eea
 generate a line bundle denoted by
$\mc L_{1,0}$ over a neighborhood of $(1,0)$. 
Our line bundle $\mc L$ coincides with the line bundle
$\mc L_{1,0}$ over a neighborhood of  $(1,0)$ in $\frak D^{(m)}_*(1)$.

\vsk.2>
Also,
the elliptic curve $X(\zz)$  has a vanishing cycle denoted by $C_{1,1}$
if the parameter $(z_1,z_2)$ is close to $(1,1)$. Then the values of the nonzero vector-function
\bea
I^{(C_{1,1})}(\zz;1)\ :=\ \int_{C_{1,1}}  \Big(\frac {1}{(t-z_1)y}\,,\ \frac {1}{(t-z_2)y}\Big)dt 
\eea
 generate a line bundle 
$\mc L_{1,1}$ over a neighborhood of $(1,1)$. 
Our line bundle $\mc L$ coincides with the line bundle
$\mc L_{1,1}$ over a neighborhood of $(1,1)$ in $\frak D^{(m)}_*(1)$.

\vsk.2>
Thus our global line bundle $\mc L$ extends over the field  $\Q^{(m)}$ the three local line bundles
$\mc L_{0,1}$, $\mc L_{1,0}$, $\mc L_{1,1}$,
each defined by integrals over the cycles vanishing  at different places.
This $p$-adic phenomenon was observed by B.\,Dwork in a different context, see \cite{Dw}  and also 
\cite[Appendix]{V2}, \cite{VZ1}. The corresponding  global line bundle was called 
a {\it $p$-cycle} in \cite{Dw}. 

\vsk.2>
The operator $\tau$ of the  \qKZ/ difference connection identifies solutions of the dynamical equations
\eqref{dy} with parameter $\la$ with solutions of dynamical equations with parameter $\la+2$.
Hence our line bundle $\mc L$ over a neighborhood of the point $(0,1)\in \frak D^{(m)}_*(1+2k)$,
$k\in\Z$, corresponds to the line bundle generated by the vector-function
\bean
\label{Ik}
I^{(C_{0,1})}(\zz;1+2k)=\int_{C_{0,1}}  \Big(\frac {1}{t^k(t-z_1)y}\,,\ \frac {1}{t^k(t-z_2)y}\Big)dt .
\eean
The vector-valued functions
\bea
I^{(C_{1,0})}(\zz;1+2k)=\int_{C_{1,0}}  \Big(\frac {1}{t^k(t-z_1)y}\,,\ \frac {1}{t^k(t-z_2)y}\Big)dt ,
\\
I^{(C_{1,1})}(\zz;1+2k)=\int_{C_{1,1}}  \Big(\frac {1}{t^k(t-z_1)y}\,,\ \frac {1}{t^k(t-z_2)y}\Big)dt .
\eea
play similar roles in neighborhoods of points $(1,0)$ and $(1,1)$, respectively.

\subsection{Monodromy}
\label{sec last}
For an odd integer $\la$, the differential operators $\mc D_i(\la)$,\, $i=1,2$, define a flat dynamical
connection on the trivial bundle
$\C^2\times \C^2 \to \C^2$.
The flat sections of the connection have the  form
\bea
I^{(C)}(\zz;\la,\mu)= \int_C
t^{-\la/2}(t-z_1)^{-1/2}
(t-z_2)^{-1/2} \Big(\frac {1}{t-z_1}, \frac {1}{t-z_2}\Big)dt,
\eea
see \eqref{IC}. The monodromy of this connection does not depend on the choice of the odd integer $\la$ and is isomorphic
to the monodromy of the Gauss-Manin connection on the bundle with fibers being the first homology groups
 of elliptic curves $X(\zz)$ of the family defined by the equation
$y^2 =t(t-z_1)(t-z_2)$. It is classically known that this monodromy is irreducible. Hence the dynamical connection 
defined by $\mc D_i(\la)$,\, $i=1,2$, over the field of complex numbers has no invariant line
subbundles. Thus the presence of our line subbundle 
$\mc L  \,\to\,  \frak D^{(m)}[\La]$,  invariant with respect to the dynamical and  \qKZ/ connections,
is a specific $p$-adic  feature.

\end{document}